\documentclass[12pt,reqno]{amsart}
\usepackage{amsmath,amssymb,amsthm}
\usepackage[mathscr]{eucal}
\usepackage{xcolor}
\def\R{\mathbb R}
\def\C{\mathbb C}

\def\F{\mathscr F} %frame bundle
\def\Fon{\F_{\operatorname{on}}}

\def\CP{\mathbb{CP}}
\def\Gr{{\mathrm{Gr}}}
\def\CC{\mathscr C} % was \mathcal

\def\restr{\negthickspace \mid}
\def\&{\wedge}

\def\realpart{\operatorname{Re}}

\def\bdet{\mathsf{det}}

\def\Q{\mathcal Q}
\def\cR{\mathcal R}
\def\cB{\mathcal B}

\def\II{\operatorname{II}}
\def\rJ{\mathsf J} % for complex structure along M
\def\bJ{\mathbf J} % complex structure on R^8
\def\Jhat{\widehat{\bJ}}
\def\rR{\mathsf R}

\def\rP{\mathsf P}
\def\vv{\mathsf v}
\def\vw{\mathsf w}

\def\ri{\mathrm i}
\def\di{\partial}

\def\ve{\mathbf e}% tangent frame vectors
\def\et{\tilde \ve}
\def\vn{\mathbf e}% normal frame vectors
\def\vf{\mathbf f}
\def\vp{\mathbf p}

\def\k{\delta}

\def\w{\omega}

\def\H{{\EuScript H}}% standard system without tangential condition
\def\Hon{\H_{\operatorname{on}}}
\def\I{{\mathcal I}}% ideal with tangential condition
\def\Iruled{\I}
\def\Ion{\I_{\operatorname{on}}}
\def\J{{\EuScript J}}
\def\K{{\EuScript K}}

\def\Khat{\widehat{\K}}

\def\ri{\mathrm i}
\def\ghat{\widehat{\gamma}}

\newtheorem{theorem}{Theorem}

\newtheorem{cor}[theorem]{Corollary}
\newtheorem{prop}[theorem]{Proposition}
\newtheorem*{theorem*}{Theorem}
\theoremstyle{definition}

\theoremstyle{remark}
\newtheorem{remark}[theorem]{Remark}

\begin{document}
\title{Ruled Austere Submanifolds of Dimension Four}
\author{Marianty Ionel}
\address{Department of Mathematics, University of Toledo, Toledo, OH; \hskip7in  Of \indent Institute of Mathematics, Federal University of Rio de Janeiro, Brazil}
\email{mionel@utnet.utoledo.edu}
\author{Thomas Ivey}
\address{Dept. of Mathematics, College of Charleston, Charleston SC 29424}
\email{iveyt@cofc.edu}


\begin{abstract}
We classify 4-dimensional austere submanifolds in Euclidean space ruled by 2-planes.
%Austere submanifolds in Euclidean space were introduced by Harvey and Lawson in connection with their study of calibrated geometries.
The algebraic possibilities for second fundamental forms of austere 4-folds $M$ were classified by Bryant, falling into three types which we label A, B, and C.
We show that if M is 2-ruled of Type A, then the ruling map from $M$ into the Grassmannian of 2-planes in $\R^n$ is holomorphic, and we give a construction for $M$ starting with a holomorphic curve in an appropriate twistor space.
If M is 2-ruled of Type B, then M is either a generalized helicoid in $\R^6$ or the product of two classical helicoids in $\R^3$.
If M is 2-ruled of Type C, then M is either a one of the above, or a generalized helicoid in $\R^7$.
We also construct examples of 2-ruled austere hypersurfaces in $\R^5$ with degenerate Gauss map.
\end{abstract}

\date{\today}
\maketitle

\section{Introduction}
A submanifold in Euclidean space $\R^n$ is {\em austere} if the eigenvalues of
its second fundamental form, in any normal direction, are symmetrically arranged around zero.
Thus, austere submanifolds are a special case of minimal submanifolds.
Harvey and Lawson \cite{HL} showed that austere submanifolds can be used to construct
special Lagrangian submanifolds in $T^*\R^n \cong \C^n$ which are ruled by their intersections
with the fibres.  Austere submanifolds of dimension 3 were completely classified by Bryant
\cite{Baustere}, but the case of dimension 4 is open.  In this paper we are concerned
with classifying austere submanifolds of dimension 4 which are themselves ruled by 2-planes.
(As we will see below in Corollary \ref{threeruled}, austere 4-manifolds ruled by 3-planes are easy to classify,
while on the other hand we expect the classification of austere submanifolds ruled by lines to be much more difficult.)

One family of ruled austere submanifolds defined in \cite{Baustere} are simple to describe.  A {\em generalized
helicoid} $M\subset \R^n$ is an $m$-dimensional submanifold swept out by $s$-planes rotating around a fixed axis,
with $n=m+s$ and $s<m$.  It can be parametrized by
\begin{multline}\label{heliparm}
(x_0, \ldots, x_{m-1}) \mapsto \\
(\lambda_0 x_0, x_1 \cos (\lambda_1 x_0), x_1 \sin (\lambda_1 x_0),
\ldots, x_s \cos (\lambda_s x_0), x_s \sin (\lambda_s x_0), x_{s+1}, \ldots, x_{m-1}),
\end{multline}
where $\lambda_0, \ldots, \lambda_s$ are constants with $\lambda_1,\ldots, \lambda_s$ nonzero.
Of course, this gives the classical
helicoid minimal surface when $s=1$ and $m=2$.  Notice also that if $m>s+1$ then this splits
as a product of a Euclidean factor with an `irreducible' helicoid swept out by $s$-planes in $\R^{2s+1}$.
Furthermore, when $\lambda_0=0$, $M$ is a cone over an austere submanifold in $S^{m-1}$ which
is ruled by $(s-1)$-dimensional totally geodesic spheres.

Helicoids (classical and generalized) will play a significant role in the latter part of this paper.
Among austere submanifolds, they have the following characterization:

\begin{theorem*}[Bryant \cite{Baustere}]  If $M \subset \R^n$ is austere, and $|\II|$ is {\em simple}
(i.e., at each point $p\in M$ the quadratic forms in $|\II_p|$ share a common linear factor) then $M$ is congruent to a generalized helicoid.
\end{theorem*}

Here, $|\II_p| \subset S^2 T^*M$ is the subspace spanned by the second fundamental form $\II$ in various normal directions.
In more detail, we recall that the second fundamental form of a submanifold $M \subset \R^n$ is
a tensor defined by
$$D_X Y = \nabla_X Y +\II(X,Y),$$
where $X,Y$ are tangent vector fields on $M$, $D$ is the Euclidean connection on $\R^n$ and the right-hand side is split
into the tangential and normal components.  Thus, $\II$ is a section of $S^2 T^*M \otimes NM$, where $NM$ is the normal bundle,
and $|\II_p|$ is the image of $\II$ under contraction with a basis for $N^*_pM$.
Let $\delta$ denote $\dim |\II_p|$, which will be constant on an
open dense set of each connected component of $M$.
We will assume $M$ is connected, and refer to the constant $\delta$ as the {\em normal rank}
of $M$.

Using an orthonormal basis to identify $T_pM$ with $\R^4$, we see that $|\II_p|$ must correspond
to a subspace of the space $S^2 \R^4$ of $4\times 4$ symmetric
matrices which is an {\em austere subspace}, in which every matrix has eigenvalues symmetrically arranged around zero.
The {\em maximal} austere subspaces $\Q \subset S^2 \R^4$ were determined (up to $O(4)$ conjugation) by Bryant \cite{Baustere},
and fall into three types which we have labeled as $\Q_A$, $\Q_B$ and $\Q_C$.
We say that an austere submanifold is $M$ is of type A,B,C
according to which one of these maximal subspaces $|\II_p|$  lies in, given an appropriate choice of orthonormal basis for $T_pM$.
We make the blanket assumption that the type is constant on an open subset of $M$, with
the exception of the case where $\delta=1$.  In that case, the notion of `type' breaks down,
as we can describe
all 1-dimensional austere subspaces of $S^2 \R^4$ up to conjugation, using diagonalization.

We now summarize our results.
The particular forms for $\Q_A, \Q_B, \Q_C$ will be given later; for now, we recall that
austere submanifolds of Type A carry a complex structure, denoted by $\rJ$, with
respect to which the metric on $M$ is K\"ahler whenever $\delta \ge 2$ (cf. Corollary 7 in \cite{ayeaye}).
In each of the following theorems,
we assume that $M$ is a austere submanifold in $\R^n$,
ruled by 2-planes, and $E_p \subset T_p M$ denotes the 2-dimensional subspace tangent to the ruling.

\begin{theorem}  If $M$ is Type A, then $E=\rJ(E)$, $\delta \le 4$,
and the ruling map $\gamma_E: M \to \Gr(2,n)$ is holomorphic.
If $\delta=2$ or $\delta=4$, then there is a compact complex manifold $V$ which fibers
over $\Gr(2,n)$, and a holomorphic mapping $\Gamma: M\to V$, which factors $\gamma_E$, and whose image is a
complex curve  $\CC$.  Moreover, any generic curve $\CC$ is locally the image of
an austere submanifold $M$ with these properties.
\end{theorem}

It is important to note that, when $n$ is even, the submanifolds $M$ here
are not in general holomorphic submanifolds of $\C^{n/2}$; our analysis shows that the space
of local solutions is too large for this to be the case.

The particular complex manifolds $V$ are defined in \S\ref{typeApart}
below for various values of $\delta$.
We refer to these as `twistor spaces',
with the justification that they are lower-dimensional homogeneous spaces in which we can take
solutions of well-known or canonical systems of PDE (e.g., the Cauchy-Riemann equations,
in the case of holomorphic curves) and use them to construct austere 4-manifolds with
the desired properties.

\begin{theorem}  If $M$ is Type B, then $\delta \le 2$.  If $\delta=2$, then $M$ is either a generalized helicoid
in $\R^6$, or the product of two classical helicoids in $\R^3$, or is also of Type A.  (In the middle case,
 the product need not be orthogonal.)
\end{theorem}

\begin{theorem}\label{Ctheorem}  If $M$ is Type C and $\delta=3$ then $M$ is a generalized helicoid
in $\R^7$ (corresponding to $s=3$ in \eqref{heliparm}).
If $\delta=2$ then $M$ is also of Type B.
\end{theorem}

We will establish these results using the techniques of moving
frames and exterior differential systems.  After setting up our basic tools
in \S2, we discuss types A, B, and C in sections 3,4,5 respectively.
Finally in \S6 we consider the case of 2-ruled austere submanifolds of normal
rank one, showing that they must lie in $\R^5$, and providing an analogous
twistor construction for this case.

\section{The standard system}\label{stdsection}
Let $\F$ be the semi-orthonormal frame bundle whose fiber at a point $p$ consists of all bases
$(\ve_1,\ve_2, \dots, \ve_n)$ of $T_p\R^{n}$ such that the vectors $(\ve_i)_{i=1\dots 4}$ are orthonormal
and orthogonal to the $(\ve_a)_{a=5\dots n}$. For the rest of the paper, we will use the index ranges
$1\leq i,j,k\leq 4$ and $5\leq a,b,c\leq n$. We choose adapted frames $(\ve_1,\ve_2, \dots, \ve_n)$
along a submanifold $M^4\subset \R^n$ such that $\ve_1(p), \cdots, \ve_4(p)$ are an orthonormal basis
of the tangent space $T_pM$ for each point $p\in M$. Note that we will not necessarily choose the vectors
$\ve_a$ to be orthonormal, as later on we will adapt them accordingly. The frame vectors are regarded as
$\R^{n}$-valued functions on the frame $\F$. We let
\begin{equation}\label{denustreqs}
\begin{aligned}
d{\vp} &= \ve_i \w^i + \ve_a \omega^a,\\
d\ve_i &= \ve_j \omega^j_i + \ve_a \omega^a_i \quad \quad  1\leq i,j\leq 4, \quad  5\leq a,b\leq n,\\
d\ve_a &= \ve_j\omega^j_a + \ve_b \omega^b_a,
\end{aligned}
\end{equation}
define the canonical 1-forms $\omega^i, \omega ^a$ and the connection 1-forms $\omega_i^j, \omega^a_i, \omega^i_a$ and $\omega_b^a$
on $\F$.
This forms span the cotangent space of $\F$ at each point, but they are not linearly independent.
Differentiating the equations  $\ve_i\cdot \ve_j=0$ and $\ve_i\cdot \ve_a=0$ yields  the relations
\begin{equation}\label{symmetries}
\omega_i^j=-\omega _j^i, \qquad  \omega_a^i=-\omega_i^b  g_{ba},
\end{equation}
 where $g_{ab}=\ve_a\cdot \ve_b$. Differentiating the first line of (\ref {denustreqs}), we obtain the structure equations
\begin{equation}\label{canonicalstreqs}
\begin{aligned}
 d\omega^i&=-\omega_j^i\wedge \omega^j+\omega^b_i g_{ba}\wedge\omega^a\\
 d\omega^a&=-\omega^a_i\wedge \omega^i-\omega^a_b\wedge \omega^b
\end{aligned}
\end{equation}
Differentiating the last two equation of   (\ref {denustreqs}) yields
\begin{equation}\label{dconn}
\begin{aligned}
d\omega^i_j &= -\omega^i_k\wedge \omega^k_j +\omega_i^b g_{ba}\wedge\omega_j^a,\\
d\omega^a_i &= -\omega^a_j\wedge \omega^j_i -\omega_b^a \wedge\omega_i^b,\\
d\omega^a_b &= \omega^a_i\wedge \omega^c_i g_{cb} -\omega_a^c \wedge\omega_b^c
\end{aligned}
\end{equation}
along with
\begin{equation}\label{diffg}
d g_{ab}= g_{ac}\omega^c_b+g_{bc}\omega^c_a.
\end{equation}
\indent An adapted frame along a submanifold $M \subset \R^n$ gives a section of $\F\vert_M$.
The following fundamental fact characterizes these sections:
\begin{quote}
A submanifold $\Sigma^4 \subset \F$
is the image of an adapted frame along the submanifold  $M^4 \subset \R^{4+\delta}$ if and only if
$\w^1 \& \w^2 \& \w^3 \& \w^4\restr_\Sigma \ne 0$ and $\omega^a\restr_\Sigma=0.$
\end{quote}
The first of these conditions is a non-degeneracy assumption called the {\em independence condition}.
The second condition implies, by differentiation, that
$\omega^a_i \restr_\Sigma = S^a_{ij} \w^j$
for some functions $S^a_{ij}$.  These functions give the components of
the second fundamental form in this frame, i.e.,
\begin{equation}\label{noose} \II(\ve_i, \ve_j) = S^a_{ij} \ve_a.
\end{equation}

In this paper we are classifying the austere $4$-folds of type A, B and C which are ruled by $2$-planes.
We will now describe a class of exterior differential systems, for later use,
whose integral submanifolds are adapted frames along 2-ruled austere submanifolds.
(Being an integral submanifold of an EDS $\I$ means that the pullback to the submanifold
of any 1-forms in $\I$ is zero).

First, suppose we wish to construct an austere submanifold $M^4$ in $\R^n$ such that at each point $p$,
$|\II_p|$ is conjugate to an austere subspace $\Q_\lambda$ of dimension $\k$ (i.e, $M$ is of type $\Q$) depending on parameters
$\lambda^1, \ldots, \lambda^\ell$ which are allowed to vary along $M$.
Let the symmetric matrices $S^5(\lambda), \ldots, S^{\k+4}(\lambda)$ be a basis for
the subspace $\Q_{\lambda}$ and suppose the parameters are allowed to range over an open set $L \subset \R^\ell$. Let
\begin{equation}\label{defoftai}
\theta_i^a:=\omega^a_i - S^a_{ij}(\lambda)\w^j.
\end{equation}
Then on $\F\times L$ we define the Pfaffian exterior differential  system
$$\H = \{ \omega^a, \theta^a_i\}$$
whose integral submanifolds correspond to austere manifolds of type $\Q_{\lambda}$.
Namely, given any austere manifold $M$ of this kind,
we may construct an adapted frame along $M$ such that
\begin{equation}\label{sis}
\II(\ve_i,\ve_j) =S^a_{ij}(\lambda)\ve_a
\end{equation}
for functions $\lambda^1, \ldots, \lambda^\ell$ on $M$.  Then the
image of the fibered product of the mappings
$p \mapsto (p, \ve_i(p), \ve_a(p))$ and $p\mapsto (\lambda^1(p), \ldots, \lambda^\ell(p))$ will be an integral submanifold of $\H$.
Conversely, any integral submanifold of $\H$ satisfying the independence
condition gives (by projecting onto the first factor in $\F\times L$)
a section of $\F\restr_M$ which is an adapted frame for an austere manifold $M$.

Now consider the additional condition that $M$ is ruled by $k$-dimensional planes.
We will let $E\subset TM$ denote a smooth distribution on $M$ with fiber $E_p$
at $p\in M$.  The following result characterizes
those distributions that are tangent to a ruling of $M$:
\begin{prop} Let $E$ be an arbitrary smooth $k$-plane field on $M$.
Then $M$ is ruled by $k$-planes in $\R^n$ tangent to $E$ if and only if at each $p\in M$
\begin{equation}\label{Druling}
{\mathsf D}_\vv \vw \in E \quad \forall\, \vv\in E_p, \vw \in C^\infty(E),
\end{equation}
where $\mathsf D$ denotes the Euclidean connection in $\R^n$.
\end{prop}

Projecting both sides  of \eqref{Druling} into the normal bundle, we get
\begin{equation}\label{IIalg}
\II(\vv,\vw)=0\quad \forall\, \vv,\vw \in E.
\end{equation}
and we see from \eqref{sis} that this puts extra conditions on  $S^a_{ij}$.
In fact, for $k=3$ the condition \eqref{IIalg} implies the following:
\begin{cor}\label{threeruled}  Any 4-dimensional austere submanifold that
is ruled by 3-planes is a generalized helicoid.
\end{cor}
\begin{proof}  If we choose an orthonormal frame in which ${\ve_1,\ve_2,\ve_3}$ spans the ruling,
then \eqref{IIalg} implies that $|\II |$ lies in the space of matrices with nonzero entries
only in the fourth row and column.  Thus, $|\II |$ is simple, and the result follows by Bryant's
theorem quoted in \S1.
\end{proof}

From now on we will consider submanifolds ruled by 2-planes.  We will ensure that
the Pfaffian system $\H$ encodes the condition \eqref{IIalg} by assuming that $E$ has
a certain basis with respect to the orthonormal frame on $M$, and choosing
the space $\Q_\lambda$ so that this condition holds for those basis vectors.
However, encoding the {\em tangential} part of \eqref{Druling} requires additional 1-form generators.

Suppose that $\vv_1,\vv_2$ are vector fields spanning $E$, $\phi_1,\phi_2$
are 1-forms that annihilate $E$, and $\vw_1, \vw_2$ are vector fields
that span the orthogonal complement of $E$ at each point in an open set $U\subset M$.
Then the tangential part of \eqref{Druling} is equivalent to
\begin{equation}\label{Pruling}
\vw_i\cdot \nabla \vv_j \equiv 0 \mod \phi_1, \phi_2
\end{equation}
for all $1\le i,j \le 2$.  (Here, $\nabla$ denotes the Riemannian connection of $M$,
so that $\nabla \vv_i$ is a $(1,1)$ tensor on $U$.)
We may encode this condition by defining 1-forms
\begin{equation}\label{psigeneral}
\psi^i_j := \vw_i \cdot \nabla \vv_j -p^i_{jk} \phi_k,
\end{equation}
for arbitrary coefficients $p^i_{jk}$.  (We will give the specific
forms for the $\psi^i_j$ in later sections.)

Thus, we define the {\em ruled austere system} $\I$ as the Pfaffian system
 generated by the forms $\omega^a$ and
$\theta^a_i=\omega^a_i - S^a_{ij}(\lambda)\w^j$ from $\H$, plus the extra
1-forms $\psi^i_j$.
The integral submanifolds of this differential ideal
\begin{equation}\label{idealI}
 \I=\{ \omega^a, \theta^a_i, \psi_j^i\}
\end{equation}
correspond to 2-ruled austere manifolds of type $\Q_\lambda$.
The system $\I$ is in general defined on the manifold $\F\times L \times \R^8$, due to
the introduction of the $p^i_{jk}$ as new variables.  However, for specific
types we will restrict to submanifolds of $\F\times L \times \R^8$ on which certain necessary integrability
conditions hold.

%We note that $d\omega^a\equiv 0$ modulo the 1-forms of $\I$.  Using \eqref{dconn},
%we obtain
%\begin{equation}\label{gen2forms}
%d\theta^a_i \equiv -\pi^a_{ij}\wedge\omega^j
%\end{equation}
%modulo the 1-forms of $\I$, where
%\begin{equation}\label{pis}
%\pi^a_{ij} =dS^a_{ij}-S^a_{kj}\omega_i^k-S^a_{ik}\omega_j^k+\omega_b^aS^b_{ij}
%\end{equation}
%Equation \eqref{gen2forms} and $d\psi_i ^j$ give the algebraic generator 2-forms of the ruled austere system.

\section{Type A}\label{typeApart}
Recall that an austere 4-fold of Type A carries a complex structure $\rJ$.  We take orthonormal
frames on $M$ with respect to which the complex structure is represented by
$$J = \begin{bmatrix} 0 & -1 & 0 & 0 \\ 1 & 0 & 0 & 0 \\ 0 & 0 & 0 & -1 \\ 0 & 0 & 1 & 0\end{bmatrix}.$$
The maximal austere subspace $\Q_A$ consists of symmetric matrices which anticommute with
$J$.  This subspace is preserved by the action of $U(2)\subset SO(4)$, the group of orthogonal
matrices commuting with $J$.
In Corollary 7 of \cite{ayeaye} we showed that any austere 4-fold of Type A
with $\delta \ge 2$ is K\"ahler with respect to $\rJ$.  Consequently, $\rJ$ is parallel along $M$
and the connection forms must satisfy
\begin{equation}\label{Kahlerconn}
\w^4_2 = \w^3_1, \qquad \w^4_1 = -\w^3_2.
\end{equation}

Suppose a type A austere submanifold is 2-ruled.
Then $E + \rJ(E)$ is $\rJ$-invariant, so is either 4- or 2-dimensional.  We treat these
cases separately.

\subsection*{Case A.1: $E=\rJ(E)$}  Let $M$ be a 2-ruled austere submanifold of this type.
Using the $U(2)$ symmetry, we can choose near each point a semiorthonormal frame with respect to
which $|\II|$ lies in $\Q_A$ and $E$ is spanned by $\ve_1, \ve_2$.
Thus, $|\II|$ must lie inside the subspace
\begin{equation}\label{arrdef}\cR = \left\{
\begin{bmatrix} 0 & 0 & a_1  & b_1  \\ 0 & 0 & b_1 & -a_1 \\ a_1 & b_1 & a_2 & b_2
 \\ b_1 & -a_1 & b_2 &-a_2  \end{bmatrix}
\right\},\end{equation}
consisting of matrices in $\Q_A$ satisfying the algebraic condition \eqref{IIalg}.
(Thus, $\delta \le 4$.)
Moreover, the ruled condition \eqref{Pruling}, together with \eqref{Kahlerconn},
implies that
\begin{equation}\label{uvxycof}
\w^3_1 =\w^4_2= u\w^3+v\w^4, \qquad \w^3_2 = -\w^4_1 = x\w^3 + y\w^4,
\end{equation}
for some functions $u,v,x,y$.

\begin{prop}\label{Grassprop} The map $\gamma_E:M \to \Gr(2,n)$, taking $p\in M$ to the
subspace parallel to $E_p$, is holomorphic.
\end{prop}
\begin{proof}Define a map $\pi:\F \to \Gr(2,n)$ that takes
$(p,\ve_1, \ldots, \ve_n)$ to $\{\ve_1, \ve_2\}$; then
$\gamma_E$ is locally the composition of $\pi$ with the lift $f:M \to \F$
provided by the frame.
The pullbacks under $\pi$ of the (1,0)-forms on $\Gr(2,n)$ are spanned by
$\w^3_1 -\ri \w^3_2, \ldots, \w^n_1 -\ri \w^n_2$
(see \cite{KN}, Chapter XI, Example 10.6).  Meanwhile, a basis for the $(1,0)$-forms
on $M$ is given by
\begin{equation}\label{firsttauzeta}
\tau := \w^1 + \ri \w^2,\qquad \zeta := \w^3+ \ri \w^4.
\end{equation}
For $a>4$, the connection forms for the framing satisfy $\w^a_i = S^a_{ij} \w^j$
for some matrices $S^a$ taking value in $\cR$.  We then compute that
%\begin{equation}\label{grass1}
$\w^a_i - \ri \w^a_2 = (S^a_{13}-\ri S^a_{14})\zeta$
%\end{equation}
along the lift of $M$ into $\F$.  Thus, $f^*(\w^a_1 - \ri \w^a_2)$ is a multiple of
$\zeta$ for each $a>4$.  We must also show that
this is also true for the pullbacks of $\w^3_1 -\ri \w^3_2$ and $\w^4_1 -\ri\w^4_2$.

The subgroup $U(1) \times U(1)$ preserves $\cR$.  Given any nontrivial
subspace of $\cR$, we may use this symmetry to arrange that
the subspace contains a matrix of the form
\begin{equation}\label{Sform}
\begin{bmatrix} 0 & 0 & p & 0 \\ 0 & 0 & 0 & -p \\ p & 0 & 0 & q \\ 0 & -p & q & 0\end{bmatrix}
\end{equation}
for $p,q$ not both zero.  Thus, along $M$ we can adapt a semiorthonormal frame
such that, say, $S^5$ has the form \eqref{Sform}.

The 1-forms $\theta^a_i$, defined by \eqref{defoftai}, vanish on $\Sigma=f(M)$.
Computing their exterior derivatives gives
$$-d\theta^a_i \equiv (dS^a_{ij}-S^a_{ik} \w^k_j -S^a_{kj}\w^k_i +S^b_{ij}\w^a_b)\& \w^j
\quad\mod \w^a, \theta^a_i.$$
In particular, using \eqref{uvxycof} we have
$$0=(-d\theta^5_1+\ri d\theta^5_2) \& \zeta = 2p(x-v+\ri(y+u))\tau \& \w^3 \& \w^4.$$
The 3-form $\tau \& \w^3 \& \w^4$ is nonvanishing since $\Sigma$ satisfies the
independence condition.  Thus, wherever $p\ne0$ we have
\begin{equation}\label{A1firstcond}
v=x, \qquad u = -y.
\end{equation}
Even if $p$ is identically zero on an open set in $\Sigma$, we can compute
$$0=(-d\theta^5_3 + \ri d\theta^5_4)\& \zeta =2q(u+y+\ri (v-x))\tau \& \w^3 \& \w^4,$$
again showing that \eqref{A1firstcond} must hold along $\Sigma$.  Substituting
\eqref{A1firstcond} into \eqref{uvxycof} shows that
$$\w^3_1 -\ri \w^3_2 = -\ri(x-\ri y)\zeta, \qquad \w^4_1 -\ri \w^4_2 =-(x-\ri y)\zeta.$$
Thus, $\gamma_E$ is holomorphic.
\end{proof}

In the remainder of this subsection, we will analyze the solutions of $\Iruled$
for each possible value of $\delta$.
%(Note that since the prolongation $\cR^{(1)}$ is nonzero, we cannot
%conclude that the effective codimension of $M$ is equal to $\delta$ in all cases.)
For the `extra' 1-form generators defined
by \eqref{psigeneral}, we may take $\phi_1=\w^3$ and $\phi_2=\w^4$.  In addition,
because of the integrability condition \eqref{A1firstcond} holds on all solutions,
we will use
\begin{equation}\label{A1extra}
\begin{aligned}
\psi^1_1 &= \w^3_1 +y \w^3 - x \w^4, \quad & \psi^1_2 &= \w^3_2 - x\w^3 - y\w^4,\\
\psi^2_1 &= \w^4_1 + x\w^3 + y\w^4, \quad & \psi^2_2 &=\w^4_2 +y \w^3 - x \w^4.
\end{aligned}
\end{equation}

\medskip
(a) $\boldsymbol{\delta=2}$.
The analysis in section 5 of \cite{ayeaye} shows that when $M$ is
an austere 4-fold of Type A, with $\delta=2$, then $|\II|$ has a two-dimensional
nullspace $E$ only if either the Gauss map is degenerate, or $M$ lies in $\R^6$ and belongs
to the set of ruled submanifolds of ``type 2.b'' in the terminology of \cite{ayeaye}.
In the first case, $M$ belongs to the class of {\em elliptic} austere submanifolds investigated
by Dajczer and Florit \cite{DF}.
%\aside{Can we say if any of these are ruled?}

In the second case, we may choose an orthonormal
frame $(\ve_1, \ldots, \ve_6)$ along $M$ such that
$$\ve_5 \cdot \II = \begin{bmatrix} 0 & 0 &r & 0 \\ 0 & 0 & 0 & -r \\ r& 0 & 0 & b\\ 0 & -r & b &0\end{bmatrix},\quad
\ve_6 \cdot \II = \begin{bmatrix} 0 & 0& 0 & r \\ 0 & 0 &r & 0\\ 0 & r & a & 0  \\ r & 0 & 0 & -a\end{bmatrix}$$
for some functions $a,b,r$ on $M$.
By regarding the frame vectors as columns in a matrix, we define a mapping
$\ghat: M \to SO(6)/U(3)$ which is rank 2 and holomorphic.  (See Theorem 15 in \cite{ayeaye}; note also
that $V=SO(6)/U(3)$ is biholomorphic to $\CP^3$.)
Conversely, given a generic holomorphic curve $\CC$ in the twistor space $V$
we may, by solving a first-order system of PDE, construct an austere submanifold $M\subset \R^6$
of this type, such that $\ghat(M)=\CC$
(see Theorem 16 in \cite{ayeaye} for more details).

\medskip
(b) $\boldsymbol{\delta=3}$.
Using the subgroup $U(1)\times U(1)$ preserving $\cR$, we can choose a semiorthonormal
frame along $M$ so that \eqref{sis} holds with
$$
S^5 = \begin{bmatrix} 0 & 0 & p & 0 \\ 0 & 0 & 0 & -p \\ p & 0 & q & 0 \\ 0 & -p & 0 &-q \end{bmatrix},\
S^6 = \begin{bmatrix} 0 & 0 & 0 & 1 \\ 0 & 0 & 1 & 0\\ 0 & 1 & 0 & 0\\ 1 & 0 & 0 & 0\end{bmatrix},\
S^7 = \begin{bmatrix} 0 & 0 & 0 & 0  \\  0 & 0 & 0 & 0  \\ 0 & 0 & 0 & 1\\ 0 & 0 & 1 & 0 \end{bmatrix}\
$$
for some functions $p,q$ on $M$,
and $S^a=0$ for $a > 7$.  In fact, one can calculate that the prolongation of $|\II|$ has dimension zero;
thus, without loss of generality we can assume that $\delta$ is the
effective codimension, and $M$ lies in $\R^7$ (see Prop. 3 in \cite{ayeaye}).

We analyze the structure equations of the standard system $\Iruled$ with
`extra' 1-forms given by \eqref{A1extra}.  In particular, we compute
$$d(\psi^1_1 - \psi^2_2) \equiv 2 p g_{56}(\w^1\&\w^4 + \w^2\&\w^3)+ (g_{66}-p^2 g_{55})(\w^2 \& \w^4 - \w^1 \& \w^3)$$
modulo $\theta^a, \theta^a_i, \psi^i_j$.%checked,
Thus, integral elements satisfying the independence condition
exist only at points where  $g_{56}=0$ and $g_{66}=p^2 g_{55}$ (with $p$ nonzero
since the matrix $g_{ab}$ must be positive definite).
When we restrict to the submanifold where these integrability conditions hold,
the system $\Iruled$ becomes involutive, with last nonzero Cartan character $s_1=12$.
%\aside{It may be possible to find a twistor space for this case; the EDS indicates
%that the map to the 6-plane spanned by e1,e2,e3,e4, e5, e6, with complex structure,
%has rank 2.}
In what follows, we will show how such submanifolds can be constructed by beginning
with a holomorphic curve in a homogeneous complex manifold.

First, note that the integrability conditions imply that we may construct an orthonormal
frame $(\ve_1, \ldots, \ve_4, \et_5, \et_6, \et_7)$ along $M$ such that $\et_5$ and
$\et_6$ are multiples of $\ve_5$ and $\ve_6$, respectively, and with respect to this frame
the second fundamental form is represented by
\[\label{Stildes}
\tilde S^5 = \begin{bmatrix} 0 & 0 & p & 0 \\ 0 & 0 & 0 & -p \\ p & 0 & q & a \\ 0 & -p & a &-q \end{bmatrix},\
\tilde S^6 = \begin{bmatrix} 0 & 0 & 0 & p \\ 0 & 0 & p & 0\\ 0 & p & 0 & b\\ p & 0 & b & 0\end{bmatrix},\
\tilde S^7 = \begin{bmatrix} 0 & 0 & 0 & 0  \\  0 & 0 & 0 & 0  \\ 0 & 0 & 0 & c\\ 0 & 0 & c & 0 \end{bmatrix}\
\]
for some functions $a,b,c,p,q$ with $c,p\ne 0$.  (Note that $p,q$ have been scaled
up by the length of the original vector $\ve_5$.)  For convenience, we drop the tildes from now on.
Let $\Ion$ denote the Pfaffian system on $\Fon \times L$ for which this orthonormal
gives an integral submanifold; that is, $\Ion$ is generated by $\w^a$ for $a=5,6,7$, $\theta^a_i:=\w^a_i-S^a_{ij}\w^j$
with $S^a$ given by \eqref{Stildes}, and $\psi^i_j$.  (Here, $L$ is the parameter
space with coordinates $a,b,c,p,q,x,y$.)

At each point $p\in M$, let $F_p$ be the 6-dimensional oriented subspace in $\R^7$
spanned by the oriented basis $\ve_1, \ldots, \ve_6$, and let $\Jhat_p$ be
the complex structure on $F_p$ taking $\ve_1$ to $\ve_2$, $\ve_3$ to $\ve_4$
and $\ve_5$ to $\ve_6$.  (This extends the complex structure on $T_p M$.)
Define $V$ as the space of triples $(E,F,\Jhat)$ where $E,F$ are
oriented subspaces of $\R^7$ and $\Jhat$ is an orthogonal complex
structure on $F$ preserving $E$.
(By adapting orthonormal frames to each triple $(E,F,\Jhat)$, we identify $V$ with the homogeneous space $SO(7)/SO(2)\times U(2)$.)
Then we define a smooth mapping $\Gamma:M\to V$ sending
$p$ to $(E_p, F_p, \Jhat_p)$.

We will give $V$ a complex structure so that $\Gamma$ is holomorphic.
To this end, define the following complex-valued 1-forms on $SO(7)$:
$$\eta^p :=\w^p_1-\ri \w^p_2, \quad \nu^p := \w^p_3 -\ri \w^p_4 \, \quad \sigma :=\w^7_5 - \ri \w^7_6.$$
Then
\begin{equation}\label{Vsemibasic}
\eta^3, \ldots, \eta^7, \nu^5-\ri \nu^6, \nu^7, \sigma
\end{equation}
are semibasic for the quotient map $\pi:SO(7)\to V$, and in fact these
span the pullback to $SO(7)$ of the bundle of $(1,0)$-forms for the complex structure on $V$.
To see that $\Gamma$ is holomorphic, we express the generators of $\Ion$ in terms
of the complex 1-forms:
\begin{equation}\label{Ioncomplex}
\begin{aligned}
\psi^1_1 -\ri \psi^1_2 &= \eta^3 + (y+\ri x)\zeta,  & \psi^2_1 -\ri \psi^2_2 &= \eta^4 + (x-\ri y)\zeta,\\
\theta^5_1 -\ri \theta^5_2 &= \eta^5 -p\zeta, & 	\theta^5_3 -\ri \theta^5_4 &= \nu^5 -p\tau +(\ri a-q)\zeta,\\
\theta^6_1 -\ri \theta^6_2 &= \eta^6 +\ri p \zeta, & \theta^6_3 -\ri \theta^6_4 &=\nu^6 +\ri p\tau +\ri b \zeta,\\
\theta^7_1 -\ri \theta^7_2 &= \eta^7, &  \theta^7_3 -\ri \theta^7_4 &= \nu^7 +\ri c \zeta,\\
\end{aligned}
\end{equation}
where $\tau, \zeta$ are as in \eqref{firsttauzeta}.  It follows that the pullbacks of the $(1,0)$-forms on $V$, when
restricted to integral submanifolds of $\Ion$, are multiples of the $(1,0)$-form $\zeta$ on $M$.

The image $\CC$ of $\Gamma:M \to V$ is not a generic holomorphic curve; in fact, it is an integral of a
well-defined complex rank 3 Pfaffian system on $V$.  To see this, first note that
 among the semibasic forms listed in \eqref{Vsemibasic}, the following
subsystem vanishes on $M$:
$$\Khat = \{ \eta^3-\ri \eta^4, \eta^5-\ri \eta^6, \eta^7\}.$$
Then one calculates that
\begin{align*}
d(\eta^3-\ri \eta^4) &\equiv (\nu^5 - \ri \nu^6) \& \eta^5, \\
d(\eta^5-\ri \eta^6) &\equiv -(\nu^5 - \ri \nu^6)\& \eta^3, \\
d\eta^7 &\equiv \eta^3 \& \nu^7 + \eta^5 \& \sigma
\end{align*}
modulo the 1-forms of $\Khat$.  The fact that the right-hand sides are pure wedge products
of semibasic forms for $\pi$ indicates that $\Khat$ is the pullback of a well-defined
Pfaffian system $\K$ on $V$ (see \cite{CFB}, Prop. 6.1.19).

\begin{remark}\label{charK} The Pfaffian system $\K$ may be characterized in two ways.

First, note that $V$ has the structure of a double fibration over more familiar complex homogeneous
spaces, the mappings $\rho_1: V \to \Gr(2,\R^7)$ and $\rho_2:V\to SO(7)/U(3)$ being
induced by the inclusions $SO(2)\times U(2) \subset SO(2)\times SO(5)$ and
$SO(2)\times U(2) \subset U(3)$ respectively.
%(Note that $SO(7)/U(3)$ is in turn the twistor space of $S^6$.)
Then $\K$ is spanned by the intersection of the pullbacks of
$(1,0)$-forms via $\rho_1$ and the pullbacks of $(1,0)$-forms via $\rho_2$.

Next, recall that the tangent space to $\Gr(2,\R^n)$ at $E$ is naturally
identified with $E^* \otimes \R^n/E$.  Thus, for any tangent vector $\vv\in T_p M$,
$\gamma_* \vv$ is an element of $(E_p)^* \otimes \R^7/E_p$.  However, the form of
the matrices $S^a$ in this case, and the vanishing of the 1-forms $\psi^i_j$ on $M$,
imply that $\Gamma: M \to V$ satisfies the following

\medskip
\begin{quote}
{\sl Contact Condition}:  For any $\vv\in T_pM$, $\Gamma_*(\vv)$ takes value in $(E_p)^* \otimes F_p/E_p$
and its value, as a mapping, is complex-linear with respect to $\Jhat_p$.
\end{quote}

\medskip
\noindent
In fact, any holomorphic mapping into $V$ is an integral of $\K$ if and only if it satisfies
this contact condition.
\end{remark}

\begin{theorem}\label{AR7twist}  Let $\CC$ be a holomorphic curve in $V$ which is an integral of $\K$, and
has a nonsingular projection onto $\Gr(2,\R^7)$.  Given any point $q\in \CC$ there is an
open neighborhood $U\subset \CC$ containing $q$ and an austere $M^4\subset \R^7$
which is Type A, 2-ruled with $\rJ(E)=E$, such that $\Gamma(M) = U$.
\end{theorem}

\begin{proof}Let $N\subset SO(7)$ be the inverse image of $\CC$.  Let $z$ be a local
holomorphic coordinate defined on open set $U \subset \CC$.  Then there are functions
$f_1, \ldots, f_5$ on $\rho_1^{-1}(U) \subset N$ such that
$$\eta^3 = f_1 dz, \quad \eta^5 = f_2 dz, \quad \nu^5-\ri \nu^6 = f_3 dz, \quad
\nu^7 = f_4 dz, \quad \sigma = f_5 dz.$$
(By hypothesis, $f_1$ and $f_2$ are never both zero.) Comparing these with
\eqref{Ioncomplex}, we see that to construct an integral manifold of $\Ion$ we
will need $f_2$ to be nonzero and $f_4/f_2$ purely imaginary.  By computing the action
of $SO(2)\times U(2)$ on the fiber of $\rho_1$, we see that there is a
smooth, codimension-one submanifold $N'\subset N$ where these conditions hold.

Let $W \subset \Fon \times L$ be the inverse image of $N'$ under the
projection to $SO(7)$.  The restriction of $\Ion$ to $W$ is spanned by
the $\w^a$ and the forms on the right-hand sides in \eqref{Ioncomplex}.  However, these
forms are now linearly dependent, because $\eta^4=\ri \eta^3$, $\eta^6=\ri \eta^5$ and $\eta^7=0$ on $W$, and the rest
satisfy the following relationships:
\begin{align*}
\eta^3 + (y+\ri x)\zeta &= f_1 dz + (y+\ri x)\zeta\\
\eta^5 -p\zeta &= f_2 dz - p\zeta\\
(\nu^5 -p\tau +(\ri a-q)\zeta) -\ri(\nu^6 +\ri p\tau +\ri b \zeta) &=
f_3 dz + (b-q + \ri a) \zeta\\
\nu^7 +\ri c \zeta &= f_4 dz+\ri c\zeta,
\end{align*}
In order to have $\zeta\ne 0$ on solutions, we need the linear combinations on the
right to be linearly dependent.  Thus, we restrict to the submanifold $W' \subset W$ on which
\begin{equation}\label{wprimeq}
y+\ri x = -\dfrac{f_1 p}{f_2},
\qquad b-q+\ri a = -\dfrac{f_3 p}{f_2},
\qquad c=\ri \dfrac{f_4 p}{f_2}.
\end{equation}

On $W'$, the remaining linearly independent generators of $\Ion$
are the real and imaginary parts of
$$\beta_1 := p\zeta - f_2 dz, \quad \beta_2 := \nu^5 - p\tau -b\zeta -f_3 dz$$
as well as the $\w^a$, which lie in the first derived system of $\Ion$.
Because \eqref{wprimeq}
determines $x,y,a,c,q$ in terms of $b,p$ and the functions of $f_j$, the remaining linearly
independent 1-forms on $W'$ are $\w^1, \ldots, \w^4$, $dp$, $db$, and two
out of the three forms $\w^2_1, \w^4_3$ and $\w^6_5$.  These last satisfy
a linear relation, due to the condition that $f_4/f_2$ is purely imaginary, which takes the form
\begin{equation}\label{dependency}
\w^2_1 - \w^4_3 - \w^6_5 \equiv 0 \mod \beta_1, \beta_2, \w^1, \ldots, \w^4.
\end{equation}

To test the Pfaffian system spanned by $\beta_1,\beta_2$ for involutivity, we compute on $W'$ that
\[\left.
\begin{aligned}
d\beta_1 &\equiv \pi_1 \& \zeta, \\
d\beta_2 &\equiv -\pi_1 \& \tau -\pi_2 \& \zeta \\
\end{aligned}
\right\} \mod \w^a, \beta^1, \beta^2
\]
where
\begin{align*}
\pi_1 &\equiv dp + \ri p (\w^6_5 - \w^4_3 -\w^2_1) \quad &\mod &\tau,\zeta \\
\pi_2 &\equiv db - 2\ri b \w^4_3 + \ri( b+p(f_3/f_2))\w^6_5 \quad &\mod &\zeta, \overline{\zeta}.
\end{align*}
In view of the relation \eqref{dependency}, the real and imaginary
parts of $\pi_1, \pi_2$ will be linearly independent provided that the coefficient
$b+p(f_3/f_2)$ is nonzero.  Accordingly, we restrict our attention to the open set
where this is the case.  Then the Pfaffian system is involutive with Cartan character
$s_1=4$.  The exist of the integral manifold of $\Ion$ inside $W'$, which projects
to the austere submanifold $M\subset \R^7$, follows by application of the Cartan-K\"ahler Theorem.
\end{proof}

The starting ingredient in Theorem \ref{AR7twist} is an integral curve of system $\K$ on the
homogeneous space $V$.  Because $\K$ is generated by holomorphic 1-forms on $V$,
this system is equivalent to an underdetermined system of three ordinary differential equations in suitable local coordinates.
Hence, it is possible that explicit solutions could be written down in terms of 4 arbitrary holomorphic functions.

\medskip
(c) $\boldsymbol{\delta=4}$.
In this case, $|\II |$ is all of the space $\cR$ defined in \eqref{arrdef}.  In this section, we will outline
a twistor-type construction for austere submanifolds $M\subset \R^n$ of this type, similar to that described just above.
However, because the prolongation of $\cR$ is nonzero, we cannot assume that the codimension of $M$
is equal to $\delta$.

Let $\Iruled$ be the standard system with the following choice
of basis for $\cR$:
$$S^5 = \begin{bmatrix} 0 & 0 & 1 & 0\\ 0 & 0 & 0 & -1\\ 1 & 0 & 0 & 0 \\ 0 & -1 & 0 & 0 \end{bmatrix},
S^6 = \begin{bmatrix} 0 & 0& 0 & 1 \\ 0 & 0 & 1 & 0 \\ 0 & 1 & 0 & 0 \\ 1 & 0 & 0 & 0 \end{bmatrix},
S^7 = \begin{bmatrix} 0 & 0 & 0 & 0\\ 0 & 0 & 0 & 0\\ 0 & 0 & 1 & 0 \\ 0 & 0 & 0 & -1 \end{bmatrix},
S^8 = \begin{bmatrix} 0 & 0 & 0 & 0\\ 0 & 0 & 0 & 0 \\ 0 & 0 & 0 & 1\\ 0 & 0 & 1 & 0 \end{bmatrix}.$$

As in the $\delta=3$ case, by differentiating the $\psi^i_j$ we derive integrability conditions;
in this case, these imply that $|\ve_5|=|\ve_6|$ and
$\ve_5 \cdot \ve_6=0$.
Once these conditions are added to the system, it is involutive with Cartan character $s_1 = 4n-16$,
indicating that solutions depend on a choice of $s_1$ functions of one real variable.

To understand the twistor construction, we first observe that the integrability conditions
 imply that to each point $p\in M$ we can associate a 6-dimensional
subspace $F_p \subset \R^n$, spanned by $\ve_1, \ldots, \ve_6$, and an orthogonal
complex structure $\Jhat_p$ on $F_p$, taking $\ve_5$ to $\ve_6$, which extends the complex structure
on $T_pM$.  We may thus define a mapping
$\Gamma:M \to V$ sending $p$ to the triple $(E_p, F_p, \Jhat_p)$, a point in the homogeneous
space $V = SO(n)/SO(2) \times U(2) \times SO(n-6)$.

Then, exactly as before, the form of $\II$ and
the vanishing of the 1-forms $\psi^i_j$ imply that $\Gamma$ satisfies the contact condition
from Remark \ref{charK}.  Furthermore, this means that the image of $\Gamma$ is an integral of a certain complex contact system $\K$ on $V$.

We define the complex structure and contact system on $V$ as follows.  Define the projections
$\rho_1: V \to SO(n)/SO(2)\times SO(n-2)$, the Grassmannian $\Gr(2,n)$, and
$\rho_2: V \to SO(n)/U(3) \times SO(n-6)$.  On $SO(n)$, define the complex-valued 1-forms
$$\eta^p := \w^p_1 - \ri \w^p_2, \qquad \nu^p:= \w^p_3 - \ri \w^p_4, \qquad \sigma^p = \w^p_5-\ri \w^p_6.$$
Then the $(1,0)$-forms for the usual complex structure on $\Gr(2,\R^n)$ are
precisely those that pull back to $SO(n)$ to be in the span of the $\eta^p$ for $p>2$.
Likewise, we define a homogeneous complex structure on the flag manifold $SO(n)/U(3) \times SO(n-6)$
by specifying that the $(1,0)$-forms pull back to $SO(n)$ to give the span
$$\{ \eta^3-\ri \eta^4, \eta^5-\ri \eta^6, \nu^5-\ri \nu^6,
 \eta^7, \nu^7, \sigma^7, \ldots, \eta^n, \nu^n, \sigma^n \}.$$
The complex structure on $V$ is uniquely specified by requiring $\rho_1, \rho_2$ to be holomorphic,
and the Pfaffian system $\K$ is defined as the intersection of the pullback of the $(1,0)$-forms under
$\rho_1$ and the pullback of the $(1,0)$-forms under $\rho_2$.
(Thus, $\K$ has complex rank $n-4$.)
Moreover, holomorphic integrals of $\K$ are exactly those mappings into $V$ that satisfy the contact condition.

With this machinery in place, we can state the analogue of Theorem \ref{AR7twist} for this case.

\begin{theorem}\label{AR8twist}
Let $\CC$ be a holomorphic curve in $V$ which is an integral of $\K$,
and has a nonsingular projection onto $\Gr(2,\R^n)$.  Given
any point $p\in \CC$, there is an open neighborhood $U\subset \CC$ containing $p$
and an austere  $M^4\subset \R^n$ which is Type A, 2-ruled satisfying $\rJ(E)=E$, such
that $\Gamma(M) = U$.
\end{theorem}

The proof is completely analogous to that of the earlier theorem; at the last stage, the Cartan-K\"ahler
Theorem is required to construct an integral for an involutive system with character $s_1=4$.

%\begin{remark}  In the case when $M$ is substantial in $\R^8$, an alternate twistor construction
%is possible, replacing the space $V$ with its quotient $V'=SO(8)/SO(2)\times U(3)$, regarded
%as the space of pairs $(E,\Jhat)$ where $\Jhat$ is an orthogonal complex structure on $\R^8$ preserving $E$.
%We use mappings $\rho_1:V' \to \Gr(2,\R^8)$ and $\rho_2:V' \to SO(8)/U(4)$
%in exactly the same way to define the complex structure and contact system $\K'$ (of rank 3) on $V'$,
%and find a holomorphic integral curve of $\K'$ in $V'$ leads to an austere 4-fold of the same type.
%In particular, any holomorphic integral curve of $\K'$ inside $V'$ has a lifting as a holomorphic
%integral of $\K$ inside $V$.
%\end{remark}

\subsection*{Case A.2: $E \cap \rJ(E) = 0$}  In this case we can use the $U(2)$ symmetry
to choose orthonormal frames so that $E$ is spanned by $\ve_1$ and $\ve_2 + k\ve_3 + m\ve_4$
for some functions $k,m$.
Then the subspace $\cR \subset \Q_A$ of matrices satisfying \eqref{IIalg} is 3-dimensional.
One can check that $\cR^{(1)}=0$, so that $\delta = \dim |\II |$ is the effective codimension of $M$.

Theorem 15 in \cite{ayeaye} implies that any 2-ruled submanifolds of Type A with $\delta=2$
are those with $\rJ(E)=E$, discussed in Case A.1(a) above.
In the case where $\delta=3$,
an extensive analysis of the exterior differential system $\Iruled$
yields integrability conditions which imply that no such submanifolds exist.

\section{Type B Ruled Austere Submanifolds}
In this section we discuss the $2$-ruled austere $4$-folds of Type B.
Recall from \cite{ayeaye} that the maximal austere subspace of Type B is given by
\begin{equation}\label{defofQB}
\Q_B = \left\{ \begin{bmatrix} m & 0 & b_1 & b_2 \\ 0 & m & b_3 & b_4 \\ b_1 & b_3 & -m &0 \\
b_2 & b_4 & 0  & -m \end{bmatrix} \right\},
\end{equation}
and may be characterized as the span of a matrix representing a reflection $\rR$ that
fixes a 2-plane in $\R^4$, together with the symmetric matrices that commute with that reflection.
Depending on the position of the ruling plane $E$ relative to the eigenspaces of $\rR$,
the intersection $\rR(E)\cap E$ is a vector space of dimension equal to 0, 1 or 2. We examine these cases separately.

\medskip
\subsection*{Case B.1:  $\dim \rR(E)\cap E=2$.}   In this case, $E$ can be an eigenspace of $\rR$
or can be a sum of the $+1$ and $-1$ eigenspaces of $\rR$.

\medskip
(a) {\bf $E$ is a eigenspace of $\rR$.} The symmetry group of $\Q_B$ is generated
by conjugation by $O(2) \times O(2)\subset O(4)$ and the permutation $\ve_1 \leftrightarrow \ve_3,
\ve_2\leftrightarrow \ve_4$. Using the permutation we can assume $E=\{\ve_1,\ve_2\}$.
Since the second fundamental form vanishes on the ruling $E$, this implies that $m=0$.
Therefore the subspace  of $\Q_B$ satisfying the algebraic condition \eqref{IIalg} is
$$\cR = \left\{\left.
\begin{bmatrix} 0 & B  \\ B & 0 \end{bmatrix}\right| B \text{ is a $2\times 2$ matrix }
\right\}.$$
Since $\cR^{(1)}$ is zero, the effective codimension of $M$ is equal to $\delta$
We break into subcases according to the value of $\delta$.
If $\delta =4$ or $\delta=3$, an analysis of the EDS $\Iruled$ shows that there
are no integral manifolds satisfying the independence condition.
In the case $\delta=2$ the analysis is more involved, and will be described in what follows.

In this case, $|\II |$ is a 2-dimensional subspace of $\cR$, and
we let $\cB$ denote the 2-dimensional of the matrices $B$ in the upper right corner.
The conjugate action of $O(2)\times O(2)$ on $2\times 2$ matrices is given by $(S,T)\cdot B=SBT^{-1}$
and the determinant is invariant up to sign under this action.
We distinguish several subcases depending
on the type of the quadratic form $\bdet$ given by the determinant restricted to $\cB$:

(a.i) {\bf $\bdet|_{\cB}$ has rank 2}.  In this case we normalize the space of second fundamental forms
so that $\cB$ is spanned by  $\left(\begin{smallmatrix} 1 & 0 \\ 0& p \end{smallmatrix}\right)$
and $\left(\begin{smallmatrix} 0 & 1 \\ q & 0 \end{smallmatrix}\right)$ where $pq<0$ if $\bdet$ is definite and $pq>0$ if $\bdet$ is indefinite.

With respect to the orthonormal basis $(\ve_1, \ve_2, \ve_3, \ve_4)$ of the tangent space to $M$, $| \II |$ is spanned by the matrices
$$S^5 = \begin{pmatrix} 0 & 0 & 1 & 0\\ 0 & 0 & 0 & p\\ 1 & 0 & 0 & 0 \\ 0 & p & 0 & 0 \end{pmatrix},\qquad
S^6 = \begin{pmatrix} 0 & 0& 0 & 1 \\ 0 & 0 & q & 0 \\ 0 & q & 0 & 0 \\ 1 & 0 & 0 & 0 \end{pmatrix}.$$
Let $\Iruled$ be the Pfaffian system, described in \S\ref{stdsection}, whose integral submanifolds
correspond to semiorthonormal frames such that
\begin{equation}\label{twoexpand}
\II(\ve_i, \ve_j) =  S^a_{ij} \ve_a,  \qquad 1\le i,j \le 4, \  5\le a\le 6
\end{equation}
The forms $\phi_1=\w^3, \phi_2=\w^4$ represent a basis of the annihilator of $E$.
The ruling condition \eqref{Pruling} amounts to requiring that
$\omega^3_1, \omega_2^3, \omega_1^4, \omega_2^4\equiv 0$ mod $\{\w^3, \w^4\}$. Therefore, the `extra' 1-forms of $\Iruled$ are
\begin{equation}\label{B1extra}
\begin{aligned}
\psi_1^1&=\omega^3_1-u_1\w^3-u_2\w^4\\\
\psi_2^1&=\omega_2^3-v_1\w^3-v_2\w^4\\
\psi_1^2&=\omega_1^4-x_1\w^3-x_2\w^4\\
\psi_2^2&=\omega_2^4-y_1\w^3-y_2\w^4
\end{aligned}\end{equation}
where $u_1, u_2, v_1, v_2, x_1,x_2, y_1, y_2$ are arbitrary functions on $M$.

\begin{prop}\label{cutdownB}
 The only austere 4-manifolds of this kind are those where $pq$ is identically equal to 1.
\end{prop}
\begin{proof}
We analyze the structure equations of the standard system $\Iruled$ with the 1-forms $\psi_j^i$ given above. We compute:
\begin{equation}\label{calcpis1}
\left.\begin{aligned}
d\theta^5_1&=\pi_1\wedge \omega^3+\pi_2\wedge \omega^4\\
d\theta^5_2&=\pi_3\wedge \omega^3+\pi_4\wedge \omega^4\\
d\theta^5_3&=\pi_1\wedge \omega^1+\pi_3\wedge \omega^2+(v_1p+x_1-2u_2)\omega^3\wedge \omega^4\\
d\theta^5_4&=\pi_2\wedge \omega^1+\pi_4\wedge \omega^2+( 2y_1p-x_2-v_2p)\omega^3\wedge \omega^4\\
d\theta^6_1&=\pi_5\wedge \omega^3+\pi_6\wedge \omega^3   \\
d\theta^6_2&=\pi_7\wedge \omega^3+\pi_8\wedge \omega^4\\
d\theta^6_3&=\pi_5\wedge \omega^1+\pi_7\wedge \omega^2+(u_1+qy_1-2qv_2)\omega^3\wedge \omega^4\\
d\theta^6_4&=\pi_6\wedge \omega^1+\pi_8\wedge \omega^2+(2x_1-u_2-qy_2)\omega^3\wedge \omega^4
\end{aligned}\right\}\qquad \mod \omega^a, \theta^a_i, \psi_j^i
\end{equation}
for certain forms $\pi_1, \dots, \pi_8$ which are linearly independent combinations
of the connection forms $\omega^i_j, \omega_b^a, dp, dq$ and $\omega^i$.
The 2-forms on the right hand side must vanish on an integral element.
Wedging the third, fourth, seventh and eighth 2-forms with $\omega^1\wedge\omega^2$
gives the following integrability conditions:
\begin{equation}\label{Bintcond}
v_1p+x_1-2u_2=0, \ 2y_1p-x_2-v_2p=0, \ u_1+qy_1-2qv_2=0, \ 2x_1-u_2-qy_2=0.
\end{equation}
Moreover, equations \eqref{calcpis1} show that all the forms $\pi_1, \dots, \pi_8$
must vanish on any integral $4$-plane also, so they have to be added to the ideal.
For example, the vanishing of the 2-form $d\theta^5_1$ implies that, on any integral element
satisfying the independence condition, $\pi_1$ must be a linear combination of $\omega^3$ and $\omega^4$,
while the vanishing of $d\theta^5_3$ implies that $\pi_1$ must be a linear combination of $\omega^1$ and $\omega^2$.
Therefore, $\pi_1=0$ on any such integral element.

Let $\J$ be the differential ideal obtained from adding the forms $\pi_1, \dots, \pi_8$
to the original ideal $\I$, and restricting to the submanifold where the integrability
conditions \eqref{Bintcond} hold (by solving for $u_1, u_2, x_1, x_2$).  We analyze this new ideal.
While $d\theta^a_i\equiv 0$ modulo the one-forms in $\J$, we also compute that:
\begin{align} \label{calcpis2}
d\psi^1_1&=(2q\pi_{10}-q\pi_{11})\wedge \omega^3+\left(\tfrac{1}{3}q\pi_{12}
+\tfrac{2}{3}p\pi_9\right)\wedge\omega^4+F_1\omega^{13}+F_2\omega^{23}+F_3\omega^{14}+F_4\omega^{24}\notag\\
d\psi^1_2&=\pi_{9}\wedge \omega^3+\pi_{10}\wedge \omega^4\\
d\psi^2_1&=\left(\tfrac{1}{3}p\pi_9+\tfrac{2}{3}q\pi_{12}\right)\wedge \omega^3+(2p\pi_{11}-p\pi_{10})\wedge \omega^4
+F_5\omega^{13}+F_6\omega^{23}+F_7\omega^{14}+F_8\omega^{24}\notag\\
d\psi^2_2&=\pi_{11}\wedge \omega^3+\pi_{12}\wedge \omega^4\notag
\end{align}
where $\pi_9, \ldots, \pi_{12}$ are
linearly independent combinations of
$dv_1, dv_2, dy_1, dy_2$ and the $\w^i$, and $F_1, \dots F_8$ are certain polynomial functions in $p,q, v_1, v_2, y_1,y_2, g_{55}, g_{56}$ and $g_{66}$.
(We also abbreviate $\omega^i\wedge \omega^j$ by $\omega^{ij}$ in the above.)

In the equations \eqref{calcpis2}, the forms $\pi_9, \pi_{10}, \pi_{11}$ and $\pi_{12}$
are not unique, but can be adjusted only by multiples of the forms $\omega^3$ and $\omega^4$;
therefore all the expressions $F_1, \dots, F_8$ must vanish at points where admissible integral elements
of the ideal $\J$ exist. These give a system of 8 linear equations in $g_{55}, g_{56}, g_{66}$
with coefficients depending on $p,q, v_1, v_2, y_1,y_2$.
Eliminating the $g_{ab}$ gives 5 homogeneous quadratic equations in $v_1, v_2, y_1, y_2$
with coefficients $p$ and $q$.  Of these, three equations constitute a homogeneous linear system in $v_1v_2, v_1y_1, y_1y_2, v_2y_2$.
The coefficient matrix of this system has rank 3, unless $pq=1$ or $pq=-1$ or $p+q=0$.
The latter cases will be considered separately later on; for now, suppose that $p+q\not =0$ and $pq\not = \pm 1$.
Then the vector $(v_1v_2, v_1y_1, y_1y_2, v_2y_2)$ must be a multiple
of $(q,q,p,p)$, which spans the kernel of the matrix.

We distinguish two possible subcases.  First,
suppose that $v_1$ and $y_2$ are identically zero on an integral submanifold.
Then $F_2=0$ implies that $g_{56}=0$, and the 2-forms \eqref{calcpis2} determine
the values of $dv_2$ and $dy_1$ uniquely on an integral element.
Setting $d^2 v_2=0$ and $d^2 y_1=0$ implies that $g_{55}=g_{66}=0$, which is impossible.
Second, if $v_1$ or $y_2$ is nonzero, then $v_1 = (q/p) y_1$ and $v_2=y_1$.
Then $F_1=\ldots=F_8=0$ implies that $g_{55}=0$, which again is impossible.

In the case that $p+q=0$, a computation of integrability conditions similar to \eqref{calcpis1}
gives that $v_1=y_2$ and $v_2=-y_1$; then, restricting to the submanifold where these relations hold and computing $d\psi^i_j$
yields contradictory integrability conditions.  We eliminate the case $pq=-1$ using similar computations.
\end{proof}

In what follows, we deal with the last case, when $pq=1$.  By making a convenient change of frame
we can assume that $|\II|$ is spanned by matrices
$$S^5 = \begin{pmatrix} 0 & 0 & 1 & p\\ 0 & 0 & 0 & 0\\ 1 & 0 & 0 & 0 \\ p & 0 & 0 & 0 \end{pmatrix},\quad
S^6 = \begin{pmatrix} 0 & 0& 0 & 0 \\ 0 & 0 & 1 & -p \\ 0 & 1 & 0 & 0 \\ 0 & -p & 0 & 0 \end{pmatrix},$$
where $p$ is some positive function on $M$.
The new basis for the normal space  is distinguished by the fact that the components
of $\II$ in the direction of each of $\vn_5, \vn_6$ has rank 2, with a 3-dimensional nullspace.
Within $E$, the lines spanned by $\ve_1$ and $\ve_2$ are distinguished as intersections with these nullspaces.
In fact, with respect to the orthonormal coframe $(\w^1, \w^2, \w^3, \w^4)$
on the tangent space, we may write
$$\II = \vn_5 \otimes \w^1 \circ \kappa_1 + \vn_6 \otimes \w^2 \circ \kappa_2,$$
where, for later convenience, we define
$$\kappa_1 = \w^3 +p\,\w^4, \qquad \kappa_2 = \w^3 - p\,\w^4.$$

Using that $d\theta^a_i\equiv 0$ mod the one-forms in the ideal $\I$ gives the integrability conditions:
$$v_1=-\frac{y_1}{p}, y_2=-py_1, u_1=\frac{x_1}{p}, x_2=px_1, u_2=x_1, v_2=y_1$$ Let us denote $x_1=x$ and $y_1=y$.
Therefore, the standard system $\I$ is defined on $\F \times \R^3$, with coordinates
$x,y$ and $p>0$ on the last factor, and the extra 1-forms are given by:
\begin{align*}
\psi^1_1 &= \w^3_1 - \frac{x}p \kappa_1, &
\psi^1_2 &= \w^3_2+\frac{y}p \kappa_2,\\
\psi^2_1 &= \w^4_1 -x\kappa_1, &
\psi^2_2 &=\w^4_2 - y\kappa_2.
\end{align*}
By computing $$(d\psi^3_1 \& \w^3 + d\psi^4_1 \& \w^4) \& \w^1$$ modulo $\I$,
we conclude that integral elements occur only at points where the following extra integrability condition holds:
\begin{equation}\label{gval}
g_{56} = xy(p^{-2}-1).\
\end{equation}
Restricting to the codimension-one submanifold $N\subset \F\times \R^3$ where this condition holds yields
an EDS $\I'$ with a unique integral element at each point.  Adjoining to $\I'$
the 1-forms that vanish on these yields a Frobenius system $\J$ on $N$, whose solutions depend on 22 constants.

We now wish to interpret the solutions.

\begin{prop}\label{heliprop}
Let $M$ be a connected, 2-ruled austere submanifold of type B, substantial
in $\R^6$, such that $E=\rR(E)$ and the rank of $\bdet |_{\cB}$ is 2.  Then $M$ is congruent to an open subset of a Cartesian
product $H_1 \times H_2 \subset \R^3 \oplus \R^3$ where $H_\alpha \subset \R^3$
is a helicoid, in which the rulings of the two helicoids are mutually orthogonal,
but the axes of the helicoids are not necessarily orthogonal.
\end{prop}

\begin{proof}
 Let $(\ve_1, \ldots, \ve_4, \vn_5, \vn_6)$ be a semiorthonormal
framing on $M$, which induces an integral $\Sigma$ of system $\J$.
If we define vectors
$$\vf_5:=\vn_5 + x(\ve_4+p^{-1}\ve_3),\quad \vf_6:= \vn_6+y(\ve_4-p^{-1}\ve_3)$$
then
$$\left.
\begin{aligned}
d\ve_1 &\equiv \vf_5 \kappa_1, & d\ve_2 &\equiv \vf_6 \kappa_2\\
d\left(\dfrac{\vf_5}{L_1}\right) &\equiv -L_1 \ve_1 \kappa_1, &
d\left(\dfrac{\vf_6}{L_2}\right) &\equiv -L_2 \ve_2 \kappa_2
\end{aligned}\right\} \mod \J,
$$
where $L_1 = |\vf_5|$ and $L_2 = |\vf_6|$.  (Note that
\eqref{gval} implies that $\vf_5$ and $\vf_6$ are orthogonal.)
This shows that the orthogonal
planes $\rP_1, \rP_2$ through the origin in $\R^6$, spanned by $\{\ve_1, \vf_5\}$ and $\{\ve_2, \vf_6\}$
respectively, are fixed, independent of the choice of point $p\in M$.  Within the rulings,
the lines spanned by $\ve_1$ and $\ve_2$ are parallel to $\rP_1$ and $\rP_2$, respectively.

In addition, we compute that $\w^1, \w^2, L_1\kappa_1, L_2 \kappa_2$ are all closed
modulo $\J$, so we may introduce coordinates $s,t,u,v$ along $M$ such that
$$ds = \w^1, \quad dt = \w^2, \quad du = L_1\kappa_1, \quad dv=L_2\kappa_2.$$
We also find that
$$d\left(\dfrac{x}{p L_1^2}\right) \equiv \w^1, \qquad d\left(\dfrac{y}{p L_2^2}\right)\equiv -\w^2\qquad
\mod \J,$$
so that we may choose the arclength coordinates $s,t$ along the $\ve_1$- and $\ve_2$-lines
satisfying
$$x=p L_1^2 s, \qquad y = -p L_2^2 t.$$

In terms of these coordinates, the derivative of the basepoint
map $\vp:\F \to \R^6$ is
\begin{equation}\label{deep}
d\vp \equiv \ve_1 ds + \ve_2 dt + \vf_3 du + \vf_4 dv
\mod \J,\end{equation}
where
$$\vf_3 = \dfrac1{2L_1}(\ve_3 + p^{-1}\ve_4), \qquad \vf_4 = \dfrac1{2L_2}(\ve_3 -p^{-1}\ve_4).$$
Computing the projections of these vectors onto the orthogonal complements of
of $\vf_5$ and $\vf_6$ respectively, we find that
$$d\left(\vf_3 - \dfrac{s}{L_1} \vf_5\right) \equiv 0, \quad
d\left(\vf_4 - \dfrac{t}{L_2}\vf_6\right)\equiv 0\qquad \mod \J.$$
In other words, these vectors are constant along $M$.

Now consider the splitting
\begin{equation}\label{splitr}
\R^6 = \{\ve_1, \vf_3, \vf_5  \} \oplus \{\ve_2, \vf_4, \vf_6\},
\end{equation}
and let $\pi_1, \pi_2$ be the projections onto the fixed $\R^3$ summands on the right.
Then \eqref{deep} shows that $\pi_1\restr_M$ and $\pi_2\restr_M$ are rank two,
with coordinates $s,u$ and $t,v$ respectively on the images.
In fact, the images are open subsets of classical helicoids.  For example, on the surface $\pi_1(M)$
the vectors $\ve_1$ and $\vf_3$ span the tangent space of the surface,
and are tangent to the $s$- and $u$-coordinate curves respectively.
The $s$-coordinate curves are straight lines.
Since $p, L_1, L_2$ are functions of $s,t$ only along $M$,
$|\vf_3|$ is constant along the $u$-coordinate curves.  We compute
$$\dfrac{\di}{\di u} \vf_3 = -s\ve_1, \quad \dfrac{\di}{\di u} \ve_1 = \dfrac{1}{L_1}\vf_5,$$
which shows that the $u$-coordinate curves are helices with
curvature $s/|\vf_3|$ and torsion $1/|\vf_3|$.

We note that the splitting \eqref{splitr} is not orthogonal, unless $p$ is identically equal to 1.
In fact, the inner product of the fixed vectors pointing along
the axes of the helicoids is given by
$$\left(\vf_3 - \dfrac{s}{L_1} \vf_5\right) \cdot \left(\vf_4 - \dfrac{t}{L_2}\vf_6\right)
=\dfrac{1-p^{-2}}{4L_1 L_2}.$$
\end{proof}

(a.ii) {\bf $\bdet|_{\cB}$ has rank 1.} In this case we can normalize so that $\cB$ is spanned by
$\left(\begin{smallmatrix} 0 & 0 \\ 1 & 0 \end{smallmatrix}\right)$
and $\left(\begin{smallmatrix} 1 & 0 \\ 0 & p \end{smallmatrix}\right)$ for $p\ne 0$.
Thus, with respect to an orthonormal basis for the tangent space, $| \II |$ is spanned by the matrices
$$S^5 = \begin{pmatrix} 0 & 0 & 0 & 0\\ 0 & 0 & 1 & 0\\ 0 & 1 & 0 & 0 \\ 0 & 0 & 0 & 0 \end{pmatrix},\qquad
S^6 = \begin{pmatrix} 0 & 0& 1 & 0 \\ 0 & 0 & 0 & p \\ 1 & 0 & 0 & 0 \\ 0 & p & 0 & 0 \end{pmatrix}.$$

Let $\I$ be the Pfaffian system for ruled austere submanifolds, with semiorthonormal frames such
that \eqref{twoexpand} holds for the above matrices $S^5, S^6$.
The 1-forms encoding the tangential ruled condition are as in \eqref{B1extra}.

Computing the 2-forms of this system yields immediate integrability conditions
$$u_2=0, \quad x_1 = -p v_1, \quad x_2 = 3p v_2, \quad y_1 = 2 v_2, \quad y_2=0.$$
We restrict to the submanifold where these conditions hold and derive further integrability
conditions on the remaining variables, which imply that $v_2=0$ and either $u_1=0$ or $v_1=0$.
In both cases, new integrability conditions arise that imply that $\ve_6 \cdot \ve_6=0$ which is impossible.
Therefore, there are no integral submanifolds in this case.

(a.iii) {\bf  $\bdet|_{\cB}\equiv 0$.} In this case we can normalize so that
$\cB= \{ \left( \begin{smallmatrix} x & 0 \\ y & 0 \end{smallmatrix}\right) \}$ or
$\{ \left( \begin{smallmatrix} x & y \\0 & 0 \end{smallmatrix}\right) \}$. In either case,  the quadratic forms in $|\II |$
have a common linear factor, hence  $|\II |$ is {\it simple}, and $M$ is a generalized helicoid.  We will determine what
are the particular values for the constants in \eqref{heliparm}.

We consider the first case, namely when $\cB$ is of a form $\{ \left( \begin{smallmatrix} x & 0\\y & 0 \end{smallmatrix}\right) \}$.
Integrability conditions imply that the extra 1-forms \eqref{B1extra} of the ideal $\Iruled$ must
have $u_2, v_2,x_1, x_2,y_1,y_2$ all zero. The first prolongation $\Iruled ^{(1)}$ of the ideal is
involutive---in fact, Frobenius, with solutions depending on 21 arbitrary constants.  Moreover, we compute that
$$d \ve_4\equiv 0 \quad \mod \Iruled^{(1)},$$
indicating that $\ve_4$ is parallel to a fixed line.  Thus, $M$ is congruent to the product of a line with a
generalized helicoid in $\R^5$ (i.e., $s=2$ in \eqref{heliparm}).

The second case is when $\cB$ is of the form $\{ \left( \begin{smallmatrix} x & y \\0 & 0 \end{smallmatrix}\right) \}$.
In this case, integrability conditions imply that all coefficients in \eqref{B1extra} vanish except for $y_2$.  Again,
the first prolongation is Frobenius, with solutions depending on 17 arbitrary constants.
However, in this case $M$ is a cone, with the vector $\ve_2$ tangent to the generators, and
is ruled by 3-planes spanned by $\{\ve_2,\ve_3,\ve_4\}$.  This corresponds to $s=3$, $\lambda_0$ and
$\lambda_1=\lambda_2=\lambda_3$ in \eqref{heliparm}.   Thus, $M$ is a cone over a generalized
Clifford torus, the embedding
of the product of the unit spheres $S^1\subset \R^2$ and $S^2 \subset \R^3$ into $S^5 \subset \R^6$.

We summarize this case as follows:
\begin{prop}  Let $M$ be a connected, 2-ruled austere submanifold of type B, substantial
in $\R^6$, such that $E=\rR(E)$ and $\delta=2$.  Then $M$ is congruent to an open subset of one of
the following: a product  of helicoids in $\R^3$,
a generalized helicoid in $\R^6$ with $s=2$ and $\lambda_0\ne 0$, or a cone over $S^1 \times S^2 \subset S^5$.
\end{prop}

\medskip
(b) {\bf $E$ is a sum of $+1$ and $-1$ eigenspaces of $\rR$.}
In this case we can assume that $E=\{\ve_1, \ve_3\}$. This implies that $m=b_1=0$ and therefore $\delta\leq 3$.
Analyzing the standard system $\Iruled$ in either the case $\delta=3$ and $\delta=2$ leads
to impossible integrability conditions.  Thus, no submanifolds of this type exist.

\subsection*{Case B.2: $\dim R(E)\cap E=1$}  Here we can arrange that $E=\{\ve_1, \ve_3+a\ve_2\}$ for $a\ne 0$.
In this case the standard system is involutive with $s_2=2$, so that
solutions depending on 2 functions of 2 variables.  However, the algebraic condition \eqref{IIalg}
implies that $|\II |$ is simple, and thus $M$ is a generalized helicoid.  In fact, it is a cone, ruled by 3-planes
spanned by $\ve_1,\ve_2, \ve_3$.  (The 2 functions come in when we choose a 2-plane field within the ruling.)
This corresponds to $s=2$ and $\lambda_0=0$ in \eqref{heliparm}, but in this instance the constants
$\lambda_1, \lambda_2, \lambda_3$ may be distinct.

\subsection*{Case B.3:  $R(E)\cap E=\{0\}$} In this case we can normalize so that
$E=\{\ve_1 +a \ve_3, \ve_2 +   \ve_4 \}$ for $a,b$ both nonzero.  It follows from \eqref{IIalg} that $\delta \leq 2$.   Assuming that $\delta=2$,
analyzing the standard system shows that solutions only exist when $a=b$.  In this case,
$|\II|$ is also of type $A$, with $\rJ(E)=E$, so these submanifolds are described in Case A.1(a) above.

%%%%%%%%%%%%%%%%%%%%%%
\section{Type C Ruled Austere Submanifolds}

In this section we discuss the ruled austere $4$-folds of Type C. We recall from \cite{Baustere}
that a maximal austere subspace of Type C is defined by
\begin{equation}\label{defofQC}
\Q_C = \left\{ \begin{bmatrix} 0 & x_1 & x_2 & x_3 \\ x_1 & 0 & \lambda_3x_3& \lambda_2x_2 \\ x_2 & \lambda_3x_3 & 0 &\lambda_1x_1 \\
x_3 & \lambda_2x_2 & \lambda_1x_1  & 0 \end{bmatrix}\right\}.
\end{equation}
where $x_1,x_2,x_3$ are linear coordinates on the subspace, and $\lambda_1, \lambda_2, \lambda_3$ are real parameters satisfying
\begin{equation}\label{lambdarel}
\lambda_1 \lambda_2 \lambda_3 + \lambda_1 + \lambda_2 + \lambda_3 = 0.
\end{equation}
An austere 4-manifold $M$ is of Type C if near any point there is an orthonormal
frame field $(\ve_1, \ve_2, \ve_3, \ve_4)$ with respect to which $|\II | \subset \Q_C$, for
parameters $\lambda_i$ which may vary smoothly along $M$.  Notice that
this condition is invariant under simultaneously permuting the frame vectors $(\ve_2, \ve_3, \ve_4)$
and the parameters $\lambda_i$.

We distinguish two cases, depending on whether or not $\ve_1$ is orthogonal to the ruling plane $E$.

\subsection*{Case C.1: $\ve_1\cdot E=0$}
In this case, the algebraic condition \eqref{IIalg} says that $E$, lying
in the span of $\ve_2, \ve_3, \ve_4$, is a 2-dimensional nullspace for every matrix in $|\II |$.  Thus, matrices
in $|\II|$ of the form in \eqref{defofQC} must satisfy
$$0=\det   \begin{bmatrix} 0 & \lambda_3x_3& \lambda_2x_2 \\
\lambda_3x_3& 0 & \lambda_1x_1\\  \lambda_2x_2 & \lambda_1x_1&0 \end{bmatrix}
=2\lambda_1\lambda_2\lambda_3x_1x_2x_3.$$
implying that either one of $x$'s is zero, or one of the $\lambda$'s, and one of these must hold on
an open set in $M$.

In the first case, we can assume that $|\II |\subset \Q_C$ is a 2-dimensional subspace defined by $x_1=0$ at each point.
Then the ruling plane is spanned by $\{\ve_3,\ve_4\}$, and $|\II|$ is also a subspace of $\Q_B$.
Swapping $\ve_3,\ve_4$ with $\ve_1,\ve_2$, we see that $M$ falls into case B.1 above.

In the second case, we can assume that $\lambda_1$ is identically zero along $M$.
Assuming first that $\delta=3$, an analysis of the standard system $\Iruled$ shows that
the remaining parameters must be constant along $M$.  Then Prop. 14 in \cite{ayeaye} implies that
$\lambda_1=\lambda_2=\lambda_3=0$ and $M$ is a generalized helicoid, swept out by 3-planes in $\R^7$.
On the other hand, if $\delta=2$ then  $E=\{\ve_3,\ve_4\}$ and the coordinates $x_1, x_2,x_3$ are linearly
related at each point.  An analysis of the system $\Iruled$ in this case, similar to that in the proof of Prop. \ref{cutdownB},
shows that no such manifolds exist.

\subsection*{Case C.2: $\ve_1\cdot E\not=0$} In this case, by permuting the frame vectors,
we can assume that $E$ is spanned by $\ve_1+a_3\ve_3+a_4\ve_4$ and $\ve_2+b_3\ve_3+b_4\ve_4$
for some functions $a_3,a_4, b_3, b_4$. The algebraic condition  \eqref{IIalg} implies that $\delta\leq 2$.
We assume that $\delta =2$.  Requiring that
a 2-dimensional subspace of $\Q_C$ satisfy the algebraic condition leads to the following subcases:\\
(i) $a_3=a_4=0$, $b_3\ne 0$, $b_4\ne 0$.  In this case, \eqref{lambdarel} implies that $\lambda_1=\lambda_2=\lambda_3=0$.
Then a change of basis shows that $|\II|$ is of Type B  and simple, falling into
case  B.1(a.iii) above; in particular $M$ is a generalized helicoid cone in $\R^6$.\\
(ii) $a_3=a_4= b_4= 0$, $b_3\not = 0$, $\lambda_3=0$.  An analysis of $\Iruled$ shows that
the only possibility is that $\lambda_1,\lambda_2$ vanish identically, in which case $M$ is again
a generalized helicoid cone.\\
(iii) $a_3=a_4= b_3=b_4 = 0$.  In this case, $|\II|$ is of Type B, falling into case B.1 above.\\
(iv) $a_4=b_3=0$, $a_3\not =0, b_4\not = 0$, $\lambda_1=-1/(a_3b_4)$, $\lambda_3=-b_4/a_3$.
A lengthy analysis of the integrability conditions for this case shows that the only possibility is that
$a_3=\pm1$ and $b_4=\pm1$.  Then an orthogonal change of basis shows that $|\II|$ is precisely of
the type described in Prop. \ref{cutdownB}, and thus $M$ is a product of helicoids as described by Prop. \ref{heliprop}.

\bigskip
We summarize our results for Type C in Theorem \ref{Ctheorem} stated in the Introduction.
%\begin{theorem}  Assume that $M$ is 2-ruled of Type C.
%If $\delta=3$ then $M$ is a generalized helicoid swept out by 3-planes in $\R^7$; if $\delta=2$ then $M$ is also of Type B.
%\end{theorem}

%%%%%%%%%%%%%%%%%%%%%%%%

%%%%%%%%%%%%%%%%%%%%%%
\section{Normal Rank One}
In this section we investigate austere submanifolds $M^4$ where $|\II |$ is one-dimensional at each point,
and the Gauss map has rank two.  As we will see, this implies that $M$ is a hypersurface in $\R^5$ ruled by 2-planes; however, it is not known if a 2-ruled austere hypersurface of this dimension must have a degenerate Gauss map.

For a submanifold in Euclidean space, the kernel of the differential of the Gauss map is known as the
{\em relative nullity distribution} (see, e.g., \cite{Chen}), and consists at $p\in M$ of those
vectors $X$ such that $\II(X,Y)=0$ for all $Y\in T_pM$.  Thus, if $M^4$ is austere with a rank two Gauss map,
then we may choose an orthonormal frame so that $\ve_1, \ve_2$ span the relative nullity distribution, and
$|\II|$ is spanned by matrices of the form
\begin{equation}\label{hyperS}
\begin{bmatrix}0 & 0 & 0 & 0 \\0 & 0 & 0 & 0 \\0 & 0 & p & q \\0 & 0 & q & -p \end{bmatrix}.
\end{equation}

Accordingly, let $\Hon$ be the standard system for austere submanifolds
(described in \S\ref{stdsection}) defined on $\Fon \times \R_*^2$ with $p,q$ not both zero
as coordinates on the second factor,
and where matrix $S^5$ is given by \eqref{hyperS} and $S^6, \ldots, S^n$ are zero.
Any austere $M^4 \subset \R^n$, equipped with an orthonormal frame such that $\ve_1,\ve_2$ span the
relative nullity distribution, gives an integral of $\Hon$, and conversely.

\begin{theorem} $M$ is ruled by planes tangent to the relative nullity distribution,
and is contained in a totally geodesic $\R^5$.
\end{theorem}
\begin{proof}  The `normal part' of the ruled condition \eqref{IIalg} holds automatically for the relative
nullity distribution.  The remaining requirement \eqref{Pruling} for $\ve_1, \ve_2$ to span a ruling
is that $\w^3_1, \w^3_2, \w^4_1,\w^4_2$ be multiples of $\w^3, \w^4$ along any integral manifold.

We will use complex-valued 1-forms to compute the structure equations of $\Hon$ and other related systems.
For example, when $n=5$ the 1-forms generating $\Hon$ are $\w^5$, $\theta^5_1=\w^5_1$,  $\theta^5_2 = \w^5_2$ and
\begin{equation}\label{Honctheta}
\theta^5_3 -\ri\theta^5_4 = (\w^5_3 - \ri \w^5_4) - (p-\ri q) (\w^3+\ri \w^4)
\end{equation}
Defining complex-valued 1-forms
$$\zeta = \w^3+ \ri \w^4, \quad \eta_1 = \w^3_1 +\ri \w^4_1, \quad \eta_2 = \w^3_2 +\ri \w^4_2,$$
we compute the nontrivial 2-forms of $\Hon$  as
\begin{equation}\label{Hon2forms}
\begin{aligned}
d\theta^5_1 &\equiv \realpart[ (p-\ri q)\eta_1\& \zeta ],\\
d\theta^5_2 &\equiv \realpart[ (p-\ri q)\eta_2\& \zeta ],\\
d(\theta^5_3 -\ri \theta^5_4) &\equiv (-d(p-\ri q) + 2\ri (p-\ri q) \w^4_3) \& \zeta + (p-\ri q)(\eta_1 \& \w^1 + \eta_2 \& \w^2)
\end{aligned}
\end{equation}
modulo the 1-forms of $\Hon$.
Thus, on any integral element
\begin{equation}\label{hyperelts}
\eta_1 = T_1 \zeta, \quad \eta_2 = T_2 \zeta
\end{equation}
for some complex-valued functions $T_1, T_2$.  In particular, $M$ is ruled.

When $n>5$, the additional 1-form generators are $\w^b$ and $\w^b_i$ for $b>5$.
It is easy to check that $d\w^b, d\w^b_1, d\w^b_2$ are congruent to zero modulo the 1-forms of $\Hon$,
while
$$d(\w^b_3-\ri\w^b_4) \equiv -(p-\ri q)\w^b_5 \& \zeta.$$
We deduce that the real-valued 1-forms $\w^b_5$ must vanish on all integral elements.
From the vanishing of $\w^b_1, \ldots, \w^b_5$ it follows, using \eqref{denustreqs},
that the subspace spanned by $\ve_1, \ldots \ve_5$ is fixed, and thus $M$ contained
in a 5-dimensional plane in $\R^n$.
\end{proof}

Without loss of generality, we will restrict our attention to austere hypersurfaces in $\R^5$ for the
rest of this section.  Let $\Ion$ be the Pfaffian system obtained by adding to $\Hon$ the
1-forms encoding the tangential part of the ruling condition; these are obtained from the real and imaginary parts
of the equations in \eqref{hyperelts}.  Altogether, the 1-form generators of $\Ion$ are
\begin{multline*}
\w^5, \ \w^5_1,\  \w^5_2,\  \w^5_3-p\w^3-q\w^4,\  \w^5_4-q\w^3+p\w^4, \\
\w^3_1 - s_1 \w^3 +t_1 \w^4,\  \w^4_1 -t_1 \w^3-s_1 \w^4,\
\w^3_2 - s_2 \w^3 +t_2\w^4,\  \w^4_2 -t_2 \w^3 -s_2\w^4,
\end{multline*}
where $s_i, t_i$ are the real and imaginary parts of $T_i$, $i=1,2$.
Then $\Ion$ is defined on $\Fon \times \R^2_* \times \R^4$, with the $s_i,t_i$ added as new variables.  It is easy to check that $\Ion$ is involutive
with Cartan character $s_1=6$.  Below, we will use the involutivity to show that such hypersurfaces
exist, passing through any generic curve in $\R^5$.

\begin{remark}  An analogue of the twistor spaces $V$, used to construct various Type A.1 ruled submanifolds in \S\ref{typeApart}, exists for this hypersurface case.  Namely, let $V$ be the space of flags $E\subset F$ in $\R^5$, where the subspaces have dimensions 2 and 4 respectively.
(This manifold is a homogeneous 8-dimensional quotient of $SO(5)$.)
Given an austere hypersurface $M$ with rank 2 Gauss map,
we define a rank 2 mapping $\Gamma:M\to V$ taking $p\in M$ to $(E_p, T_pM)$, where $E_p$ is
parallel to the ruling through $p$.  Then $\Gamma(M)$ is an integral surface
of a certain exterior differential system $\J$ on $V$, and any such surface is the locally the image
$\Gamma(M)$ for some austere hypersurface $M$ of this type.  (Details are left to the interested reader.)
The main difference between this case
and the Type A twistor spaces is that there is no homogeneous complex structure on $V$ with respect
to which $\Gamma(M)$ is a holomorphic curve.
\end{remark}

In closing, we consider a Cauchy
problem for austere hypersurfaces.  Let $\gamma(s)$ be a regular real-analytic curve, parametrized by arclength, which is substantial in $\R^5$ (i.e., it lies in no lower-dimensional totally geodesic submanifold).  It follows that there exist real-analytic orthonormal vectors $T, N_1, N_2$
and real-analytic functions $k_1, k_2$ such that
$$\dfrac{d\gamma}{ds} = T, \quad \dfrac{d T}{ds} = k_1 N_1, \quad \dfrac{d N_1}{ds} = -k_1 T + k_2 N_2.$$
We lift $\gamma$ to obtain a real-analytic integral curve of $\Ion$ in $\Fon\times \R^6$ as follows.
First, letting $\ve_3 = T$, $\ve_5=N_1$, $\ve_4=N_2$, and letting $\ve_1, \ve_2$ be any analytic
orthonormal vectors along $\gamma$ that are perpendicular at each point to the plane spanned by $\{ T, N_1, N_2\}$,
gives a lift into $\Fon$ which is an integral of $\w^5, \w^5_1$ and $\w^5_2$.  (As well, $\w^1, \w^2, \w^4$ pull back to be zero along $\gamma$, while $\w^3$ pulls back to equal $ds$.)
Next, we choose values for the
functions $p,q,s_1,t_1, s_2, t_2$ along $\gamma$ so as to give a lift into $\Fon\times \R^6$ which
is an integral curve of the remaining generators of $\Ion$.  (For example, since
$\omega^5_3= \ve_5 \cdot (d\ve_3/ds)$ and $\omega^5_4= \ve_5 \cdot (d\ve_4/ds)$, we set $p=k_1$ and $q=-k_2$.)

It follows by the Cartan-K\"ahler Theorem (see, e.g., \cite{CFB}), that there exists an
integral 4-manifold of $\Ion$, satisfying the independence condition, containing the lift of $\gamma$.
Our hypersurface $M$ is the projection of this 4-manifold into $\R^5$.
Note that the fact that $N_1$ is orthogonal to the hypersurface means that $\gamma$ is a geodesic in $M$.

\begin{prop}
 Let $\gamma$ be a real-analytic curve, substantial in $\R^5$.  Then there exists an austere 4-fold $M^4\subset \R^5$ with rank 2 Gauss map, such that $\gamma$ is a geodesic in $M$ and along $\gamma$ the ruling plane is orthogonal to the span of $\gamma', \gamma'', \gamma'''$.  $M$ is unique in the sense that
 any two such hypersurfaces must coincide in a neighborhood of $\gamma$.\end{prop}

The `function count' $s_1=6$ provided by the Cartan's Test may be interpreted as follows.  Choosing
a curve in $\R^5$ (up to rigid motions) depends on four functions of one variable, namely the Frenet curvatures.
Then we use up the remaining two function's worth of freedom by choosing
a lift into $\Fon$ along which $\w^1, \w^2, \w^4, \w^5, \w^5_1, \w^5_2$ are zero
and $\ve_5$ lies in the osculating plane of $\gamma$.

\end{document}